\documentclass{amsart}
\usepackage{amssymb,amsfonts,amsthm,amsmath,amscd, xy}
\usepackage{latexsym}
\usepackage{url}
\usepackage{amsmath}
\usepackage{graphics, color}
\definecolor{darkblue}{rgb}{0,0,0.4}
\usepackage[colorlinks=true, citecolor=darkblue, filecolor=darkblue,
linkcolor=darkblue, urlcolor=darkblue]{hyperref}
\input xy
\xyoption{all}

\title{Extending properties to relatively hyperbolic groups}

\author{Daniel A. Ramras}
\address{
		Department of Mathematical Sciences\\
		Indiana University-Purdue University, Indianapolis\\ 
		402 N. Blackford Street\\
		Indianapolis, IN 46202\\
}
\email{dramras@iupui.edu}

\author{Bobby W. Ramsey}
\address{
		Department of Mathematics\\
		The Ohio State University\\
        	Columbus, OH 43210\\
        }
\email{ramsey.313@math.osu.edu}

\thanks{Ramras was partially supported by a Collaboration Grant from the Simons Foundation, USA (\#279007)}

\newtheorem*{thm*}{Theorem}
\newtheorem*{cor*}{Corollary}
\newtheorem{thm}{Theorem}[section]
\newtheorem{defn}[thm]{Definition}
\newtheorem{prop}[thm]{Proposition}
\newtheorem{cor}[thm]{Corollary}
\newtheorem{lem}[thm]{Lemma}

\newtheorem{remark}[thm]{Remark}

\newcommand{\leqs}{\leqslant}
\newcommand{\geqs}{\geqslant}
\newcommand{\N}{{\mathbb{N}}}

\newcommand{\im}{\operatorname{im}}
\newcommand{\mH}{\mathcal{H}}
\newcommand{\I}{\mathcal{I}}
\newcommand{\B}{\mathcal{B}}
\newcommand{\Y}{\mathcal{Y}}
\newcommand{\Z}{\mathcal{Z}}
\newcommand{\X}{\mathcal{X}}
\renewcommand{\P}{\mathcal{P}}

\newcommand{\D}{\mathfrak{D}}
\newcommand{\bk}{\mathbf{k}}
\newcommand{\wD}{\mathrm{w}\mathfrak{D}}
\newcommand{\sD}{\mathrm{s}\mathfrak{D}}

\newcommand{\e}{\emph}
\def\co{\colon\thinspace}
\newcommand{\srm}[1]{\stackrel{#1}{\longrightarrow}}
\newcommand{\xmaps}{\xrightarrow}


\begin{document}

\begin{abstract}
Consider a finitely generated group $G$ that is relatively hyperbolic with respect to a family of subgroups $H_1, \ldots, H_n$.
We present an axiomatic approach to the problem of extending metric properties from the subgroups $H_i$ to the full group $G$. We use this to show that both (weak) \e{finite decomposition complexity} and \e{straight finite decomposition complexity} are extendable properties.  We also discuss the equivalence of two notions of straight finite decomposition complexity. 
\end{abstract}

\maketitle

\section{Introduction}
The concept of relative hyperbolicity was proposed by Gromov in \cite{G_HypGrps}, as a generalization of hyperbolicity.  Farb, 
Bowditch, Osin, and Mineyev--Yemen, \cite{Bowd_RHG, F_RHG, Os_book, MY}, have developed this in various directions, which are
equivalent for finitely generated groups.  We follow the approach to relatively hyperbolicity given by Osin \cite{Os_book}.

Say $G$ is a finitely generated group that is relatively hyperbolic with respect to a family of subgroups $\{H_i\}_{i=1}^n$, as defined in Section~\ref{RHG-sec}.
Various authors have considered the problem of extending metric properties of the subgroups $H_i$ to the  full group $G$.  In particular, \e{finite asymptotic dimension}, \e{coarse embeddability} (also known as \e{uniform embeddability}), and \e{exactness} are all known to be extendable~\cite{Os_FiniteAsdim, DG_UERHG, ozawa_RHG}.  The main goal of this article is to show that \e{finite decomposition complexity}~\cite{GTY-rigid, GTY_FDC}   and \e{straight finite decomposition complexity}~\cite{DZ} are extendable properties.  For finite decomposition complexity, this was previously observed by Sisto~\cite{Sisto}.

In this article, we present an axiomatic approach to the problem of extendability.  This approach is similar in spirit to~\cite{Guentner}, where properties $\P$ of \e{metric families} (that is, sets of metric spaces) are studied from the point of view of permanence.
We say that a metric space $X$ has the property $\P$ if the metric family $\{X\}$  has $\P$.  
We identify several conditions that such a property $\P$ may satisfy, which together imply the extendability of $\P$ for relatively hyperbolic groups.  These conditions are Coarse Inheritance, the Finite Union Theorem, the Union Theorem, and the Transitive Fibering Theorem, which are defined in Section~\ref{extend-sec}.  The Transitive Fibering Theorem is a weak version of  Fibering Permanence from~\cite{Guentner}, and is needed for our study of straight finite decomposition complexity (variants of the other conditions also appear in~\cite{Guentner}).  We also assume that $\P$ is satisfied by all metric spaces with finite asymptotic dimension.  Our main tool for extending such properties is the work of Osin~\cite{Os_FiniteAsdim} regarding the \e{relative Cayley graph} of a relatively hyperbolic group.

Finite decomposition complexity (FDC) and its weak version (wFDC) 
were introduced in~\cite{GTY-rigid} as natural generalizations of finite asymptotic dimension (FAD), and were used to study rigidity properties of manifolds. The more general notion of straight finite decomposition complexity (sFDC) was recently introduced in~\cite{DZ}.  (We review the definitions in Section~\ref{sFDC}.)
The class of groups with FDC is already quite large, and contains all countable linear groups~\cite[Theorem 3.1]{GTY-rigid}.  By~\cite[Theorem 3.4]{DZ}, all metric spaces with sFDC satisfy Yu's Property A, so finitely generated groups  with sFDC satisfy the coarse Baum--Connes conjecture~\cite{yu}.  We have the following chain of implications relating these concepts:
$$\xymatrix{ \textrm{FAD}  \ar@{=>}[r] & \textrm{FDC} \ar@{=>}[r]   & \textrm{wFDC}  \ar@{=>}[r]  
	& \textrm{sFDC} \ar@{=>}[r] & \textrm{Property A}.}
$$
The fact that weak FDC implies straight FDC is proven in Proposition~\ref{impls}.  It is also possible to formulate a weak version of sFDC, although it was shown by Dydak and Virk that this weak version is in fact equivalent to sFDC~\cite{Dydak-Virk}. 

Our results on extendability interact nicely with recent work in algebraic $K$--theory.  It was shown in~\cite{RTY} that the Integral $K$--theoretic Novikov Conjecture (injectivity of the $K$--theoretic assembly map) holds for all group rings $R[G]$, where $R$ is a unital ring and $G$  has finite decomposition complexity and a finite classifying space $K(G, 1)$. 
 If $G$ is torsion-free and relatively hyperbolic with respect to subgroups $\{H_i\}_{i=1}^n$ satisfying the conditions of this theorem, then $G$ also satisfies the conditions: by Corollary~\ref{FDC-cor}, $G$ has finite decomposition complexity, and by~\cite[Theorem A.1]{FO-BCRH}, there exists a  finite $K(G,1)$\footnote{Kasprowski~\cite{Kasprowski} has shown that the  Integral $K$--theoretic Novikov Conjecture holds for all groups $G$  with finite decomposition complexity and a  \e{finite-dimensional} classifying space. 
Hence it would be interesting to know whether the property of having a finite-dimensional classifying space is extendable (at least for torsion-free groups).}.
In related work, Goldfarb~\cite{Goldfarb-sFDC} showed that finitely generated groups with sFDC satisfy \e{weak regular coherence}, which guarantees the existence of projective resolutions of finite length for certain $R[\Gamma]$--modules over sufficiently well-behaved coefficient rings $R$.  This work is part of a program for proving surjectivity of assembly maps~\cite{Carlsson-Goldfarb-Borel}.
 
 \vspace{.2in}

\noindent {\bf Acknowledgment.} We thank the referee for informing us of~\cite{Sisto} and for other helpful comments, and we thank Daniel Kasprowski for helpful discussions regarding weak decompositions.

\section{Relatively Hyperbolic Groups}\label{RHG-sec}
Suppose $G$ is a finitely generated group with a finite symmetric generating set $S$, and let $\left\{ H_i \right\}_{i=1}^k$
be a family of finitely generated subgroups.  Then $G$ is a quotient of the free product $F = F(S) * H_1 * H_2 * \dots * H_k$, where 
$F(S)$ is the free group on $S$.  Say that $G$ is \emph{finitely presented relative to $\left\{ H_i \right\}_{i=1}^k$} if the kernel of 
the projection $F \to G$ is the normal closure of a finite subset $\mathcal{R}$ in $F$.  (Note that if $G$ is finitely presented, then it is also finitely presented relative to $\{H_i\}_{i=1}^k$).

Set $\mH = \sqcup_{i=1}^k \left( H_i \setminus \{1\} \right)$.  If a word $w$ in the alphabet $S \cup \mH$ represents the identity
element of $G$, it can be expressed in the form $w = \prod_{j = 1}^{m} a_i^{-1} r_i^{\pm 1} a_i$ where $r_i \in \mathcal{R}$ and
$a_i \in F$ for $i = 1, \ldots, m$.  The smallest possible number $m$ in such a representation of $w$ is the \emph{relative area}
of $w$, denoted by $Area_{rel}(w)$.

\begin{defn}
$G$ is \emph{hyperbolic relative to the collection of subgroups} $\left\{ H_i \right\}_{i = 1}^{k}$ if it is finitely 
presented relative to $\left\{ H_i \right\}_{i = 1}^{k}$ and there is a constant $K$ such that every word $w$ in $S \cup \mH$
that represents the identity in $G$ satisfies $Area_{rel}(w) \leq K \| w \|$, where $\| w \|$ represents the length of the word
in $S \cup \mH$.
\end{defn}

A key construction in relatively hyperbolic groups is the relative Cayley graph, $\Gamma(G, S \cup \mH)$; that is, the Cayley
graph of $G$ with respect right multiplication by elements in    
the generating set $S \cup \mH$.  This graph is not locally finite.  However Osin has proven the
following.
\begin{thm}[{\cite[Theorem 17]{Os_FiniteAsdim}}]\label{thm:relCayfad}
	The relative Cayley graph $\Gamma(G, S\cup\mH)$ has finite asymptotic dimension.
\end{thm}

The existence of constants $L$ and $\varepsilon$ involved in the following two lemmas (from \cite{Os_FiniteAsdim}) will 
be necessary in what follows, though the results themselves will not be mentioned again.  The terminology and notation is taken from~\cite{Os_FiniteAsdim}.

\begin{lem}\label{lem:L}
Suppose that a group $G$ is generated by a finite set $S$ and is hyperbolic relative to $\left\{ H_i \right\}_{i=1}^{k}$.  Then there
is a constant $L > 0$ such that for every cycle $q$ in $\Gamma(G, S\cup \mH)$, every $i \in \{1, \dots, k\}$, and every set of isolated
$H_i$-components $p_1, \ldots, p_m$ of $q$, we have
	\[ \sum_{j=1}^{m} d_S\left( (p_j)_{-}, (p_j)_{+} \right) \leq L \| q \|. \]
\end{lem}

\begin{lem}\label{lem:epsilon}
Suppose that a group $G$ is generated by a finite set $S$ and is hyperbolic relative to $\left\{ H_i \right\}_{i=1}^{k}$.  Then
for any $s \geq 0$, there is a constant $\varepsilon = \varepsilon(s) \geq 0$ such that the following condition holds.  
Let $p_1$ and $p_2$ be two geodesics in $\Gamma(G, S\cup \mH)$ such that $d_S\left( (p_1)_{-}, (p_1)_{+} \right) \leq s$
and $d_S\left( (p_2)_{-}, (p_2)_{+} \right) \leq s$.  Let $c$ be a component of $p_1$ such that 
$d_S\left( c_{-}, c_{+} \right) \geq \varepsilon$.  Then there is a component of $p_2$ connected to $c$.
\end{lem}

\section{Extendable properties} \label{extend-sec}
Many properties can be extended from the peripheral subgroups $H_1, \ldots, H_n$ to the group $G$.
Coarse embeddability \cite{DG_UERHG}, exactness \cite{ozawa_RHG}, finite asymptotic dimension \cite{Os_FiniteAsdim},
and combability \cite{JR} are just a few examples of such properties.  An analysis of \cite{DG_UERHG} and \cite{Os_FiniteAsdim}
shows much similarity in method.

Given a countable group $\Gamma$, we will view $\Gamma$ as a metric space with respect to a proper left-invariant metric.  Any two such metrics are coarsely equivalent, and the properties under consideration here are all coarsely invariant, so the choice of metric will not matter.    

Suppose that $\P$ is some property of metric families.  We isolate a few features that may hold for $\P$, which will be of interest.
Recall that a map between metric spaces, $f \co X \to Y$, is \emph{uniformly expansive} if there exists a nondecreasing function
$\rho\co [0,\infty) \to [0,\infty)$ such that for all $x, x' \in X$, $d_Y( f(x), f(x') ) \leq \rho( d_X( x, x' ) )$. 
Such a map is \emph{homogeneous} if for all $y_1, y_2 \in \im(f) \subset Y$ there exist isometries $\phi \co X \to X$
and $\bar \phi \co Y \to Y$ such that
	\begin{itemize}
			\item $f \circ \phi = \bar \phi \circ f$, and
			\item $\bar \phi (y_1) = y_2$.
	\end{itemize}

\begin{lem}\label{homog-lemma}
Let $G$ be a finitely generated group, with finite symmetric generating set $S$, and let $\mH$ be a finite family of
subgroups.  Then the map $p \co G \to \Gamma(G, S \cup \mH)$, which sends a group element to the vertex it represents,
is homogeneous.
\end{lem}
\begin{proof}
Let $g, g' \in G$.  Denote by $v_g$ and $v_{g'}$ the vertices in $\Gamma(G, S \cup \mH)$ identified with $g$ and $g'$, respectively.  
As $p$ is equivariant with respect to left multiplication in $G$,    
 we define $\phi \co G \to G$ and $\bar \phi \co \Gamma(G, S \cup \mH) \to \Gamma(G, S \cup \mH)$
through left multiplication by the element $g' g^{-1}$.  Thus $\bar \phi \left( g \right) =  g' $, and $p \circ \phi = \bar \phi \circ p$.
\end{proof}

There are several versions of the Fibering Theorem.  We will establish the following version for straight finite decomposition complexity in Section~\ref{sFDC-sec}.  Recall that  we say a metric space $X$ has $\P$ if the family $\{ X \}$ has $\P$.

\begin{defn}[Homogeneous Fibering Theorem]\label{defn:PFibers}
Say that \emph{$\P$ satisfies the Homogeneous Fibering Theorem} if the following holds.
\begin{quote}
	Let $f\co E \to B$ be a uniformly expansive, homogeneous map.  Assume $B$ has property $\P$
        and for each bounded subset $D \subset B$, the inverse image $f^{-1}(D)$ has property $\P$.
	Then $E$ has property $\P$.
\end{quote}
\end{defn}

A significantly weaker version of the above will suffice for studying extendability.  We say that a map $f\co X\to Y$ of metric spaces is \e{contractive}, or a \e{contraction}, if $d(f(x), f(y))\leq d(x,y)$ for all $x,y\in X$.  Such maps are uniformly expansive.

\begin{defn}[Transitive Fibering Theorem]
Say that \emph{$P$ satisfies the Transitive Fibering Theorem} if the following holds.
\begin{quote}
	Let $\Gamma$ be a countable group acting isometrically on $E$ and $B$, and assume $B$ has finite asymptotic dimension and that $\Gamma$ acts transitively on $B$.
	Let $f: E \to B$ be a contractive, $\Gamma$--equivariant map.   If for each bounded subset $D \subset B$,  $f^{-1}(D)$ has property $\P$, then $E$ has property $\P$.
\end{quote}
\end{defn}

We note that the maps $p$ considered in the Transitive Fibering Theorem are automatically homogeneous, since $\Gamma$ is acting by isometries.

\begin{defn}[Finite Union Theorem]
Say that \emph{$\P$ satisfies the Finite Union Theorem} if the following holds.
\begin{quote}
Let $X$ be a metric space written as a finite union of metric subspaces $X = \cup_{i = 1}^{n} X_i$.  If each 
$X_i$ has $\P$ then so does $X$.
\end{quote}
\end{defn}

The next property addresses more general unions.  Recall that two subsets $A, B$ of a metric space $X$ are said to be $r$--disjoint if $d(A,B) > r$.

\begin{defn}[Union Theorem]
Say that \emph{$\P$ satisfies the Union Theorem} if the following holds.
\begin{quote}
Let $X$ be a metric space written as a union of metric subspaces $X = \cup_{i \in \I} X_i$.  Suppose that
$\{ X_i \}_{i \in \I}$ has $\P$ and that for every $r > 0$ there exists a metric subspace $Y(r) \subset X$ with $\P$
such that the sets $Z_i(r) = X_i \setminus Y(r)$ are pairwise $r$--disjoint.  Then $X$ has $\P$.
\end{quote}
\end{defn}

\begin{defn}[Coarse Inheritance]
Say that \emph{$\P$ satisfies Coarse Inheritance} if the following holds.
\begin{quote}
Let $X$ and $Y$ be metric spaces.  If there is a coarse embedding from $X$ to $Y$ and $Y$ has $\P$, then so does $X$.
\end{quote}
\end{defn}

Note that if $\P$ satisfies Coarse Inheritance, then it is a coarsely invariant property.

\begin{defn}
Say that $\P$ is \e{axiomatically extendable} if it satisfies the  Transitive Fibering Theorem, the Finite Union Theorem, the Union Theorem, and Coarse Inheritance, 
and every metric space with finite asymptotic dimension has $\P$.
\end{defn}

\begin{prop}
Coarse embeddability, exactness, and finite decomposition complexity $($see Definition~\ref{FDC-def}$)$ are axiomatically extendable properties.
\end{prop}

Coarse embeddability and exactness for metric families are defined in~\cite[Definitions 2.2 and 2.8]{DG_UERHG}, where they are referred to as `equi-embeddability' and `equi-exactness'.

\begin{proof}
For coarse embeddability, the Coarse Inheritance property is clear.   The Finite Union Theorem and the Union Theorem are Corollaries 4.5 and  4.6 of \cite{DG_UERHG}.  The Transitive Fibering Theorem is a special case of Corollary 4.7 of \cite{DG_UERHG}.  Finally,
spaces of finite asymptotic dimension are coarsely embeddable \cite{RoeLCG}.

We now turn to exactness.  Again, the Coarse Inheritance property follows easily from the definition.
Metric spaces of finite asymptotic dimension are exact, by Proposition 4.3 of \cite{DG_UERHG}.  The Finite Union Theorem, Union Theorem, and  Transitive Fibering Theorem come from Corollaries 4.5, 4.6, and 3.4 of \cite{DG_UERHG}.

For finite decomposition complexity, Coarse Inheritance, the Finite Union Theorem, the Union Theorem, and a stronger version of the Fibering Theorem appear in Section 
3.1 of \cite{GTY_FDC}. 
That spaces of finite asymptotic dimension have finite decomposition complexity is proven in~\cite[Section 4]{GTY_FDC}, using \cite{DZ-FAD}.
\end{proof}

\begin{thm}\label{thm:extending}
Suppose that $\P$ is an axiomatically extendable property.  If $G$ is relatively hyperbolic with respect to $H_1, \ldots, H_n$ and each $H_i$ has $\P$, 
then $G$ has $\P$.
\end{thm}

We begin by proving an auxiliary lemma.
Let 
$$B(n) = \left\{ g \in G \, : \, d_{S\cup\mH}(e,g) \leq n \right\}.$$  That is, $B(n)$ is the closed ball around $e$ of radius $n$ in $\Gamma(G, S\cup\mH)$.
We consider $B(n)$ as a metric subspace of $G$, with the word metric associated to $S$.

\begin{lem}\label{lem:BallFDC}
Suppose that each $H_i$ has $\P$.  For any  integer $n > 0$, $B(n)$ has $\P$.
\end{lem}

Note that the word metric on $G$ restricts to give a proper left-invariant metric on each $H_i$, and we will choose this as our metric on $H_i$.

\begin{proof}  The argument is based on the proof of~\cite[Lemma 3.2]{Os_FiniteAsdim}.
Proceed by induction on $n$.  
For $n = 1$, $B(1) = S \cup \left( \cup_{i=1}^{k} H_i \right)$ has $\P$ by the Finite Union Theorem.
Let $n > 1$ and assume $B(m)$ has $\P$ for all positive integers $m < n$.  We have
\[ B(n) = \left( \bigcup_{i=1}^k B(n-1) H_i \right) \cup \left( \bigcup_{s \in S} B(n-1) s \right). \]
As each $B(n-1)s$ is coarsely equivalent to $B(n-1)$ and $S$ is finite, $\bigcup_{s \in S} B(n-1) s$
has $\P$ by the Finite Union Theorem and the induction hypothesis.  It remains to check that $\bigcup_{i=1}^k B(n-1) H_i$ has $\P$.

Fix   $i \in \{1, \dots, k\}$ and let $R(n-1)$ be a subset of $B(n-1)$ such that
 \[ B(n-1) H_i = \bigsqcup_{r \in R(n-1)} r H_i. \]

Fix an $s > 0$ and set 
\[ T_s = \left\{ g \in G \, : \, d_S( e, g ) \leq \max\{ \varepsilon, 2L(s+1) \} \right\}, \]
where $L$ and $\varepsilon = \varepsilon(s)$ are the constants from Lemmas \ref{lem:L} and \ref{lem:epsilon} respectively.
Let $Y_s = B(n-1)T_s$.
As $T_s$ is finite, $Y_s$ has $\P$.  Osin shows in \cite[Lemma 3.2]{Os_FiniteAsdim} that the sets 
$\{ r H_i \setminus Y_s \, : \, r \in R(n-1) \}$ are $s$--disjoint, so $B(n-1)H_i$ has $\P$ by the Union Theorem.  The Finite Union Theorem then shows $\bigcup_{i=1}^k B(n-1) H_i$ has $\P$.
\end{proof}

\begin{proof}[Proof of Theorem \ref{thm:extending}]
Consider the map $p  \co G \to \Gamma(G, S\cup\mH)$.  This is a contraction, thus it is uniformly expansive.  
By Theorem  \ref{thm:relCayfad}, $\Gamma(G, S\cup\mH)$ has finite asymptotic dimension, so $\Gamma(G, S\cup\mH)$ has the property $\P$ as well.

For each bounded subset $Z$ of $\Gamma(G, S\cup\mH)$ there is an $n$ such that $p^{-1}(Z)$
lies in $B(n)$.  By Lemma~\ref{lem:BallFDC},     
$B(n)$ has $\P$, and $p^{-1}(Z)$ has $\P$ as well by Coarse Inheritance.  
Consider the map $p\co G\to \textrm{Image} (p)$, which is equivariant with respect to the transitive left-translation actions of $G$ (in fact, $p$ is simply the identity map on underlying set $G$).
Since  $\Gamma(G, S\cup\mH)$ has finite asymptotic dimension, so does  $\textrm{Im} (p)\subset \Gamma(G, S\cup\mH)$.
By the Transitive Fibering Theorem, 
$G$ has the property $\P$.
\end{proof}

\begin{cor}$\label{FDC-cor}$
Suppose $G$ is relatively hyperbolic with respect to $H_1, \ldots, H_n$.  If each $H_i$ has finite decomposition complexity,
 so does $G$.
\end{cor}

The same argument shows that this result holds with FDC replaced by either of the weak versions ($k$--FDC or wFDC) discussed in the next section, since the 
extendability arguments  
for FDC  in~\cite{GTY_FDC} all apply to these weak versions as well.

\section{Straight finite decomposition complexity} \label{sFDC}

We recall the definition of finite decomposition complexity from \cite{GTY_FDC}.
\begin{defn}
An \emph{$(k, r)$--decomposition} of a metric space $X$ over a metric family $\Y$ is a decomposition
\[ X = X_0 \cup X_1 \cup \cdots \cup X_{k-1} , \,\,\, X_i = \bigsqcup_{r-\textrm{disjoint}} X_{ij}, \]
where each $X_{ij} \in \Y$.  A metric family $\X$ is \emph{$(k, r)$--decomposable} over $\Y$ if every
member of $\X$ admits a $(k, r)$--decomposition over $\Y$.  
\end{defn}

We will write 
$$\X\xmaps{(k, r)} \Y$$
to indicate that $\X$ admits a $(k,r)$--decomposition over $\Y$.      
When $k=2$, we recover the notion of $r$--decomposition from~\cite{GTY_FDC}.  In this case, we will write 
\begin{equation}\label{not}\X \srm{r} \Y\end{equation}
to mean that $\X$ admits an $r$--decomposition over $\Y$.

\begin{remark} $\label{decomp-rmk}$ If $X$ admits a $(k, r)$--decomposition over a metric family $\Y$, then it also admits a $(k', r)$--decomposition over $\Y$ for each $k'\geqs k$, since we may repeat the spaces $X_i$ appearing in the decomposition $($or add copies of the empty set$)$.     
\end{remark}

\begin{defn}
Let $\mathfrak{U}$ be a collection of metric families.  A metric family $\X$ is $k$--\emph{decomposable} over $\mathfrak{U}$
if, for every $r > 0$, there is a metric family $\Y_r \in \mathfrak{U}$ and a $(k,r)$--decomposition of $\X$ over $\Y_r$.      
The collection $\mathfrak{U}$ is \emph{stable under $k$--fold decomposition} if every metric family which $k$--decomposes over
$\mathfrak{U}$ actually belongs to $\mathfrak{U}$.  

A metric family is \emph{weakly decomposable} over $\mathfrak{U}$ if it is $k$--decomposable over $\mathfrak{U}$ for some $k\in \N$.
\end{defn}

Recall that a metric family $\Z$ is \emph{uniformly bounded} if 
\[\sup \{ \textrm{diam}(Z) :  Z \in \Z \} < \infty.\]

\begin{defn}\label{FDC-def}
The collection $\D^k$ of metric families with $k$--fold finite decomposition complexity $($$k$--FDC$)$ is the smallest collection of metric families
that contains the uniformly bounded metric families and is stable under $k$--fold decomposition.  When $k=2$, we recover the notion of FDC from~\cite{GTY-rigid, GTY_FDC}.

The collection $\wD$ of metric families with weak finite decomposition complexity $($wFDC$)$ is the smallest collection of metric families
that contains the uniformly bounded metric families and is stable under weak decomposition.
\end{defn}

\begin{remark} As explained in~\cite{GTY_FDC} and in~\cite[Section 6]{RTY} for the case of FDC, the collections $\D^k$ are     
unions of collections of families $\D^k_\alpha$ indexed by countable  ordinals $\alpha$.   One starts with $\D^k_0 = \B$, the collection of uniformly bounded metric families, and then inductively defines  $\D^k_{\alpha+1}$ to be the set of metric families that $k$--decompose over $\D^k_\alpha$ $($for limit ordinals $\beta$, one may simply set $\D^k_\beta = \bigcup_{\alpha < \beta} \D_\alpha$$)$.  One then checks that the union of the union of the collections $\D^k_\alpha$, taken over all countable ordinals $\alpha$, is stable under  $k$--fold decomposition.    
 The same remark applies to $\wD$.
\end{remark}

By Remark~\ref{decomp-rmk}, we have $\D^1 \subset \D^{2} \subset \D^3 \subset \cdots \subset \wD$.

In \cite{DZ}, Dranishnikov and Zarichnyi give the following generalization of FDC, whose applications to algebraic $K$--theory
have been studied by Goldfarb~\cite{Goldfarb-sFDC}. 

\begin{defn}
A metric family $\X$ has \emph{straight finite decomposition complexity} $($sFDC$)$ if, for every sequence
$R_1 < R_2 < \ldots$ of positive numbers, there exists an $n \in \N$ and metric families $\X_0, \X_1, \X_2, \ldots, \X_n$
such that $\X = \X_0$, the family $\X_i$ is $R_{i+1}$--decomposable over $\X_{i+1}$, and the
family $\X_n$ is uniformly bounded.  
The collection of metric families with sFDC is denoted by $\sD$.

We say that a metric space $X$ has sFDC if the single-element family $\{X\}$ has sFDC.    
\end{defn}

\begin{prop}\label{impls}  Every metric family with weak $FDC$ also has straight $FDC$.  In fact, the class of metric families $\sD$ is stable under weak decomposition.
\end{prop}

We need the following lemma, which was pointed out to us by Daniel Kasprowski (personal communication).  A similar idea appears in Dydak--Virk~\cite{Dydak-Virk}.

\begin{lem}\label{decomp}    
Let $\X$ and $\Y$ be a metric families such that $\X$ admits a $(k, s)$ decomposition over $\Y$.  There there exists a sequence of decompositions
$$\X = \X_0 \srm{S} \X_1 \srm{S} \cdots \srm{S} \X_{k-1} \srm{S} \X_k = \Y.$$
\end{lem}
\begin{proof} For each $X\in \X$, there exists a $k$--fold decomposition $X = X_1 \cup \cdots \cup X_k  $, with each $X_i$ an $s$--disjoint union 
$$X_i = \bigsqcup_{j\in J_i}^{s-\textrm{disjoint}} X_{ij},$$
such that  $X_{ij}  \in \Y$ for each $i,j$. 
For $l = 1, 2, \ldots, k$, we define $\X_l$ to be the metric family consisting of all the spaces
$X_{ij}$ with $1\leqs i \leqs l$ and $j\in J_i$, together with the space $X_{l+1} \cup \cdots \cup X_k$.  In other words,
$$\X_l   = \bigcup_{X\in \X} \left( \bigcup_{i = 1}^l \{X_{ij} : j\in J_i\} \right) \cup \{X_{l+1} \cup \cdots \cup X_k\}.$$
The desired  decompositions $\X_l  \srm{S} \X_{l+1}$ are obtained by decomposing each $X_{ij}$ with $i\leqs l$ trivially, and decomposing   $X_{l+1} \cup \cdots \cup X_k$ as the union of $X_{l+1}$ and $X_{l+2} \cup \cdots \cup X_k$; we can then further decompose
$$X_{l+1} =  \bigsqcup_{j\in J_{l+1}}^{s-\textrm{disjoint}} X_{(l+1) j},$$
and we can decompose $X_{l+2} \cup \cdots \cup X_k$ trivially.
\end{proof}

\begin{proof}[Proof of Proposition~\ref{impls}]     
We will show that $\sD$ is stable under weak decomposition.  Since $\sD$ contains all uniformly bounded families and $\wD$ is the smallest  collection  of metric families that is stable under weak decomposition and contains all uniformly bounded families, this will imply that $\wD\subset \sD$.

Say $\X$ weakly decomposes over $\sD$.  Then there exists $k\geqs 1$ such that for each $r > 0$, there exists  $\Y(r) \in \sD$ such that
$$\X \xmaps{(k, r)} \Y(r).$$
By Lemma~\ref{decomp}, there exist metric families $\X_1(r), \ldots, \X_{k-1}(r)$ such that
$$\X \srm{r} \X_1(r) \srm{r} \X_2(r)\srm{r} \cdots \srm{r} \X_{k-1}(r)\srm{r} \Y(r).$$

Now consider a sequence $R_1 < R_2 < \cdots$.  Setting $r=R_k$, we have
$$\X \xmaps{R_k} \X_1(R_k)\xmaps{R_k} \X_2(R_k) \xmaps{R_k} \cdots \xmaps{R_k} \X_{k-1}(R_k)\xmaps{R_k} \Y(R_k),$$
and since $R_1, R_2, \ldots, R_{k-1} < R_k$, we in fact have
\begin{equation}\label{first}\X \xmaps{R_1} \X_1(R_k)\xmaps{R_2} \X_2(R_k) \xmaps{R_3} \cdots \xmaps{R_{k-1}} \X_{k-1}(R_k)\xmaps{R_k} \Y(R_k).
\end{equation}
Since $ \Y(R_k) \in \sD$, applying the definition of sFDC to the sequence 
$$R_{k+1} < R_{k+2} < \cdots$$ yields a finite sequence of decompositions of $\Y(R_k)$ ending with a bounded family; that is, for some $n\in \mathbb{N}$ we have
\begin{equation}\label{2nd}\Y(R_k) \xmaps{R_{k+1}} \Z_{k+1} \xmaps{R_{k+2}} \Z_{k+2} \xmaps{R_{k+2}} \cdots \xmaps{R_{k+n}} \Z_{k+n}
\end{equation}
 with $\Z_{k+n}$ uniformly bounded.  Stringing together  (\ref{first}) and (\ref{2nd}) shows that $\X$ has sFDC.
\end{proof}

The notion of sFDC can be weakened in a manner analogous to the definition of weak FDC.  

\begin{defn}
 A metric family $\X$ has \emph{weak straight finite decomposition complexity}
if there exists a sequence 
$\bk = (k_1, k_2, \ldots)$ $($$k_i\in \N$$)$ such that
 for every sequence
$R_1 < R_2 < \ldots$ of positive numbers, there exists an $n \in \N$  and metric families $\X_0, \X_1, \X_2, \ldots, \X_n$
such that $\X = \X_0$, the family $\X_{i}$ is $(k_{i+1}, R_{i+1})$--decomposable over $\X_{i+1}$, and the
family $\X_n$ is uniformly bounded.  We say that $\X$ has weak sFDC \e{with respect to the sequence} $\bk = (k_1, k_2, \ldots)$.
\end{defn}

Dydak and Virk call this notion \e{countable asymptotic dimension}.  In an earlier version of this article, we asked whether metric spaces with weak sFDC have Property A.  In fact, Dydak and Virk show that countable asymptotic dimension is \e{equivalent} to sFDC~\cite[Theorem 8.4]{Dydak-Virk}, and sFDC spaces have Property A by a result of Dranishnikov and Zarichnyi~\cite{DZ}. 

%
%

\section{Extendability of straight finite decomposition complexity}\label{sFDC-sec}

We now consider basic extendability properties for straight finite decomposition complexity.

The usual argument for coarse inheritance of FDC also proves the following result.

\begin{lem}$\label{ci-wsFDC}$
 If $X$ has sFDC and there exists a coarse embedding $Y\to X$, then $Y$ also has sFDC.  
\end{lem}

For the next result, the following notion for metric families will be useful.
\begin{defn}
Let $\X$ be a metric family.  The \emph{subspace closure} of $\X$, denoted by $\X'$
is the metric family  $\X' = \left\{ X  : \textrm{ there exists } Y \in \X \textrm{ with } X \subset Y \right\}$.
\end{defn}

\begin{thm}\label{thm:sFDCFibers}
Let $f \co E \to B$ be a uniformly expansive, homogeneous map.  Assume that $B$ has sFDC   
and assume that there exists $b_0 \in B$ 
such that for each $r > 0$, the space  $f^{-1} (B_r (b_0))$ has  sFDC. 
Then $E$ has sFDC.

In particular,  sFDC satisfies the Homogeneous Fibering Theorem.
\end{thm}
\begin{proof} 
Take $\rho$ to be the function from the definition of uniform expansion for $f$, and let $R_1 < R_2 < \ldots$ be given. 
Since $B$ has sFDC 
there is an $n \in \N$ and a sequence of metric families $\Y_0 = \{ B\}$, $\Y_1$, \ldots, $\Y_n$ such that $\Y_{i-1}$ is $\rho( R_i )$--decomposable over $\Y_{i}$
and $\Y_n$ is a uniformly bounded family.
Let \[ f^{-1}( \Y_i ) = \left\{ f^{-1}(Y) \, : \, Y \in \Y_i \right\}. \]
Then $f^{-1}(\Y_0) = \{ E \}$, and $f^{-1}( \Y_{i} )$ can be $R_{i+1}$--decomposed over $f^{-1}( \Y_{i+1} )$, since  inverse images of $\rho( R_{i+1} )$--disjoint sets in $B$ are $R_{i+1}$--disjoint in $E$.

This yields a sequence of decompositions of $E$ that ends with the family $f^{-1}( \Y_n )$, and by assumption 
there exists $r > 0$ such that each $Y \in \Y_n$ has diameter at most $r$.
Each $f^{-1}(Y)$ is isometric, via one of the isometries
$\bar \phi$ guaranteed by the definition of homogeneity, to a subspace of $f^{-1}( B_r( b_0 ) )$, so by Lemma~\ref{ci-wsFDC} we conclude that each space $f^{-1} (Y)$ has sFDC. 

Applying the definition of sFDC to the space $f^{-1}( B_r( b_0 ) )$ and the sequence of numbers $R_{n+1}< R_{n+2}< \cdots$ shows that there exists $N\geq 0$ and metric families
\[ \Z_{n}(b_0) = \left\{ f^{-1}( B_r(b_0) ) \right\}, \Z_{n+1}(b_0), \Z_{n+2}(b_0), \ldots, \Z_{n+N}(b_0)\]
such that  $\Z_{n+N}(b_0)$ is uniformly bounded and for $i=0, \ldots N-1$, $\Z_{n+i}(b_0)$ admits an $R_{n+i+1}$--decomposition over $\Z_{n+i+1}(b_0)$.

For $i = 0, \ldots, N$, let $\Z_{n+i}$ be the union over $b \in B$ of all translates of spaces in $\Z_{n+i}(b_0)$ under the isometries $\bar \phi$.  
Since decomposability is defined element-wise over elements in a metric family, we see that $\Z_{n+i}$ 
admits an $R_{n+i+1}$--decomposition 
over $\Z_{n+i+1}$.  Let $\Z'_{n+i}$ be the subspace closure of $\Z_{n+i}$, and note that $\Z'_{n+N}$ is still uniformly bounded.  
If $Z'\subset Z$ are metric spaces, then each decomposition of $Z$ can be 
intersected with $Z'$ to obtain a decomposition of $Z'$.
Hence $\Z'_{n+i}$ admits an $R_{n+i+1}$--decomposition over $\Z'_{n+i+1}$, and the same idea shows that $f^{-1}(\Y_n)$
admits an $R_{n+1}$--decomposition over $\Z'_{n+1}$.

The sequence of decompositions 
\begin{eqnarray*}\{E\} = f^{-1} (\Y_0) \srm{R_1} f^{-1}(\Y_1) \srm{R_2} f^{-1}(\Y_2)\srm{R_3}  \cdots \srm{R_n}  f^{-1}(\Y_n)\\
\xmaps{R_{n+1}}  \Z'_{n+1}\xmaps{R_{n+2}}  \ldots \xmaps{R_{n+N}} \Z'_{n+N} \end{eqnarray*}
shows that $E$ has sFDC.
\end{proof}

The Finite Union Theorem and the Union Theorem for sFDC were established in~\cite[Theorems 3.5 and 3.6]{DZ}.  
We now have the following consequence of Theorem~\ref{thm:extending}.

\begin{cor}$\label{sFDC-cor}$
Straight finite decomposition complexity  axiomatically extendable.   
In particular, if $G$ is relatively hyperbolic with respect to $H_1, \ldots, H_n$ and each $H_i$ has sFDC, then $G$ has sFDC.
\end{cor}


\begin{thebibliography}{10}

\bibitem{Bowd_RHG}
B.~H. Bowditch.
\newblock Relatively hyperbolic groups.
\newblock {\em Internat. J. Algebra Comput.}, 22(3):1250016, 66, 2012.

\bibitem{Carlsson-Goldfarb-Borel}
Gunnar Carlsson and Boris Goldfarb.
\newblock Algebraic {$K$}-theory of geometric groups.
\newblock arXiv:1305.3349, 2013.

\bibitem{DG_UERHG}
Marius Dadarlat and Erik Guentner.
\newblock Uniform embeddability of relatively hyperbolic groups.
\newblock {\em J. Reine Angew. Math.}, 612:1--15, 2007.

\bibitem{DZ-FAD}
A.~Dranishnikov and M.~Zarichnyi.
\newblock Universal spaces for asymptotic dimension.
\newblock {\em Topology Appl.}, 140(2-3):203--225, 2004.

\bibitem{DZ}
Alexander Dranishnikov and Michael Zarichnyi.
\newblock Asymptotic dimension, decomposition complexity, and {H}aver's
  property {C}.
\newblock {\em Topology Appl.}, 169:99--107, 2014.

\bibitem{Dydak-Virk}
Jerzy Dydak and Ziga Virk.
\newblock Preserving coarse properties.
\newblock {\em  Rev. Mat. Complut.}, 29(1):191-206, 2016.

\bibitem{F_RHG}
B.~Farb.
\newblock Relatively hyperbolic groups.
\newblock {\em Geom. Funct. Anal.}, 8(5):810--840, 1998.

\bibitem{FO-BCRH}
Tomohiro Fukaya and Shin-ichi Oguni.
\newblock The coarse {B}aum-{C}onnes conjecture for relatively hyperbolic
  groups.
\newblock {\em J. Topol. Anal.}, 4(1):99--113, 2012.

\bibitem{Goldfarb-sFDC}
Boris Goldfarb.
\newblock Weak coherence of groups and finite decomposition complexity.
\newblock to appear in Int. Math. Res. Not.
  \url{http://www.albany.edu/~goldfarb/papers_html/WCKGFDC-140315.pdf}, 2013.

\bibitem{G_HypGrps}
Michael Gromov.
\newblock Hyperbolic groups.
\newblock In {\em Essays in group theory}, volume~8 of {\em Math. Sci. Res.
  Inst. Publ.}, pages 75--263. Springer, New York, 1987.

\bibitem{Guentner}
Erik Guentner.
\newblock Permanence in coarse geometry.
\newblock In {\em Recent progress in general topology. {III}}, pages 507--533.
  Atlantis Press, Paris, 2014.

\bibitem{GTY-rigid}
Erik Guentner, Romain Tessera, and Guoliang Yu.
\newblock A notion of geometric complexity and its application to topological
  rigidity.
\newblock {\em Invent. Math.}, 189(2):315--357, 2012.

\bibitem{GTY_FDC}
Erik Guentner, Romain Tessera, and Guoliang Yu.
\newblock Discrete groups with finite decomposition complexity.
\newblock {\em Groups Geom. Dyn.}, 7(2):377--402, 2013.

\bibitem{JR}
Ronghui Ji and Bobby Ramsey.
\newblock The isocohomological property, higher {D}ehn functions, and
  relatively hyperbolic groups.
\newblock {\em Adv. Math.}, 222(1):255--280, 2009.

\bibitem{Kasprowski}
Daniel Kasprowski.
\newblock On the {$K$}-theory of groups with finite decomposition complexity.
\newblock {\em Proc. Lond. Math. Soc. (3)}, 110(3):565--592, 2015.

\bibitem{MY}
I.~Mineyev and A.~Yaman.
\newblock Relative hyperbolicity and bounded cohomology.
\newblock \url{http://www.math.uiuc.edu/~mineyev/math/art/rel-hyp.pdf}.

\bibitem{Os_FiniteAsdim}
D.~Osin.
\newblock Asymptotic dimension of relatively hyperbolic groups.
\newblock {\em Int. Math. Res. Not.}, (35):2143--2161, 2005.

\bibitem{Os_book}
Denis~V. Osin.
\newblock Relatively hyperbolic groups: intrinsic geometry, algebraic
  properties, and algorithmic problems.
\newblock {\em Mem. Amer. Math. Soc.}, 179(843):vi+100, 2006.

\bibitem{ozawa_RHG}
Narutaka Ozawa.
\newblock Boundary amenability of relatively hyperbolic groups.
\newblock {\em Topology Appl.}, 153(14):2624--2630, 2006.

\bibitem{RTY}
Daniel~A. Ramras, Romain Tessera, and Guoliang Yu.
\newblock Finite decomposition complexity and the integral {N}ovikov conjecture
  for higher algebraic {$K$}-theory.
\newblock {\em J. Reine Angew. Math.}, 694:129--178, 2014.

\bibitem{RoeLCG}
John Roe.
\newblock {\em Lectures on coarse geometry}, volume~31 of {\em University
  Lecture Series}.
\newblock American Mathematical Society, Providence, RI, 2003.

\bibitem{Sisto}
A.~Sisto.
\newblock Finite decomposition complexity (is preserved by relative
  hyperbolicity), 2012.\\
\newblock
  \url{http://alexsisto.wordpress.com/2012/10/03/}.

\bibitem{yu}
Guoliang Yu.
\newblock The coarse {B}aum-{C}onnes conjecture for spaces which admit a
  uniform embedding into {H}ilbert space.
\newblock {\em Invent. Math.}, 139(1):201--240, 2000.

\end{thebibliography}

\end{document}